\documentclass[11pt]{amsart}
\usepackage{amsmath, amsthm, amssymb,amsfonts}
\usepackage{verbatim,geometry,graphicx}
\usepackage{enumerate,matlab-prettifier}
\usepackage{algpseudocode,algorithm}

\theoremstyle{definition}

\theoremstyle{plain}
\newtheorem{theorem}{Theorem}

\newtheorem{lemma}[theorem]{Lemma}

\newtheorem{rem}[theorem]{Remark}

\newcommand{\R}{{\mathbb R}}
\newcommand{\C}{{\mathbb C}}
\newcommand{\nxn}{{n\times n}}
\newcommand{\diag}{\operatorname{diag}}

\newcommand{\bmat}[1]{\begin{bmatrix}#1\end{bmatrix}}
\newcommand{\abs}[1]{\vert #1 \vert}
\newcommand{\lam}{\lambda}

\title{Forming a symmetric, unreduced, tridiagonal matrix with a given spectrum}

\author{Luca Dieci}
\address{School of Mathematics \\ Georgia Tech \\
	Atlanta, GA 30332 U.S.A.}
\email{dieci@math.gatech.edu}
\author{Alessandro Pugliese}
\address{Dipartimento di Matematica, Univ. of Bari, I-70100, Bari,
	Italy}
\email{alessandro.pugliese@uniba.it}
\thanks{
The authors wish to thank Cinzia Elia for having provided the motivation
for this note and the School of Mathematics of Georgia Tech for hosting
the visit of Alessandro Pugliese.  The work has been partially supported by the 
GNCS-Indam group.}


\begin{document}

\maketitle

\begin{abstract}
Given a set of $n$ distinct real numbers, our goal is to form a 
symmetric, unreduced, tridiagonal, matrix with those numbers
as eigenvalues.  We give an algorithm which is a
stable implementation of a naive algorithm forming the characteristic 
polynomial and then using a technique of Schmeisser.
\end{abstract}

\pagestyle{myheadings}
\thispagestyle{plain}
\markboth{L. Dieci, A. Pugliese}{Unreduced tridiagonal}

\section{The problem}
We are interested in solving the following inverse problem, Problem 1.
\begin{itemize}
\item 
Problem 1: {\emph{Given $n$ real values
$\lambda_1 < \lambda_2 < \cdots < \lambda_n$, $n\gg 1$, find an unreduced,
symmetric, tridiagonal matrix $A$, with these values as eigenvalues}}.   \\
Mathematically, via the association of the characteristic polynomials of $A$
to its roots, this problem
is equivalent to one examined by Schmeisser in \cite{schmeisser1993real}.

\item 
Problem 2:  {\emph{Given a monic
polynomial $p(x)$, with distinct real roots 
$\lambda_1<\lambda_2 < \cdots < \lambda_n$, 
find an unreduced, symmetric, tridiagonal matrix $T$, with the given polynomial
as characteristic polynomial}}.
\end{itemize}
In \cite{schmeisser1993real}, the author gave a very elegant solution of the
above Problem 2, which is ultimately an algorithmic procedure for building the matrix
$T$.  So, in principle, Problem 1 could be solved in two steps by the following
algorithm.

\begin{algorithm}
	\caption{Naive algorithm to solve Problem 1}\label{simple}
	\begin{algorithmic}[1]
		\State Form the polynomial $p(x)$ needed in Problem 2, from its roots
		$\lambda_1<\lambda_2 < \cdots < \lambda_n$.
		\State Use the procedure of \cite{schmeisser1993real}, Algorithm 
		\ref{algo:schm} below, to obtain $T$.  
	\end{algorithmic}
\end{algorithm}

Of course, a  symbolic implementation of Algorithm \ref{simple}
renders $T$ exactly, but unfortunately a symbolic implementation limits us
to using small sized matrices, which is undesirable.
Regretfully, when
implementing Algorithm \ref{simple} in finite precision, 
we obtain very disappointing
results insofar as the quality of the eigenvalues of the computed $T$.
This is actually not surprising, since already forming 
$p(x)$ from the $\lambda_j$'s, say using the {\tt Matlab} command {\tt poly},
is an unstable process (this fact is well known, see \cite{Beckermann}).
In this short note, we propose a stable algorithm to solve Problem 1,
and we prove that our algorithm is {\emph{mathematically equivalent}} to the above
two-step strategy of Algorithm \ref{simple}.

\begin{rem}
We must appreciate that the scope of the algorithm in \cite{schmeisser1993real} is to
obtain a matrix $T$ whose characteristic polynomial is the given $p(x)$.  And, in this
respect, the algorithm of \cite{schmeisser1993real} performs, in finite precision,
as well as one may wish for, that is the characteristic polynomial of the computed $T$
has coefficients that are accurate approximation to the coefficients of the 
original $p(x)$.  It is
the eigenvalues of the computed $T$ that are not accurate approximations
of the given $\lambda_i$'s.
\end{rem}

\smallskip
\noindent {\bf Notation}. 
We let $e_k$ be the $k$-th column of the identity matrix, and $e$ be the vector of
all $1$'s. 

\section{The algorithm of \cite{schmeisser1993real}, our algorithm, 
and its equivalence to Algorithm \ref{simple}}
\label{algos}

Here, we first review the technique of \cite{schmeisser1993real}, then, we
propose our new algorithm, Algorithm {\tt diag2trid} in \eqref{algo:diag2trid}.
In Figure \ref{fig:spec_accu},
we show the performance, in finite precision, of Algorithm \ref{simple} and of
Algorithm \ref{algo:diag2trid}, 
on a suite of problems with random real eigenvalues uniformly distributed
in  $[-10,10]$; data are averaged over $100$ realizations for each value of $n$.
Accuracy is measured by 
considering the worse absolute error between  each prescribed eigenvalue and the
corresponding eigenvalue of the computed tridiagonal matrix.   After $n\approx 26$,
there is not even a single digit of accuracy.

\begin{figure}[h]
\centering
\includegraphics[width=0.66\textwidth]{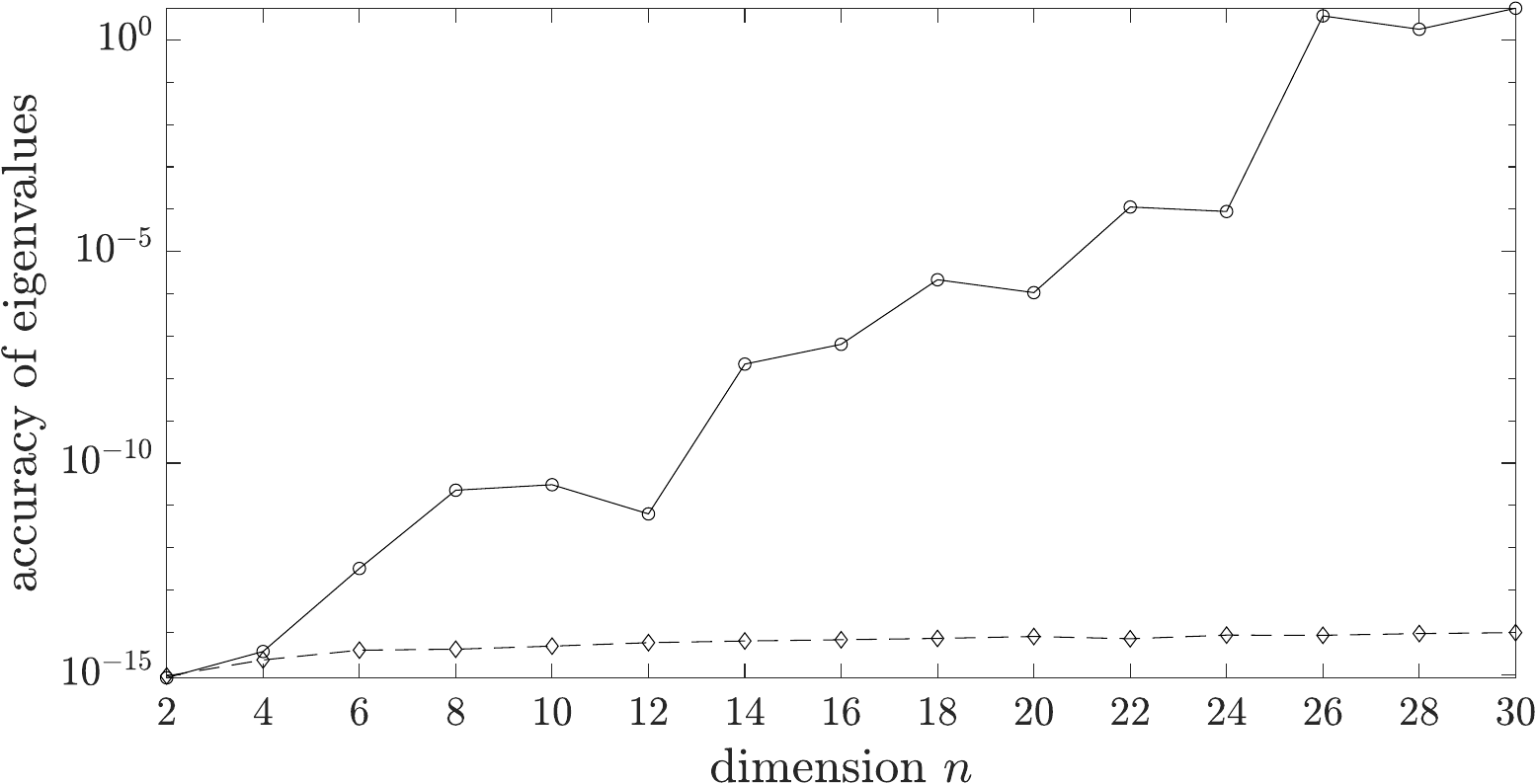}
\caption{The figure shows the accuracy of the eigenvalues of the tridiagonal matrix
computed through Algorithm \ref{simple} (solid-circles) and Algorithm \ref{algo:diag2trid} 
(dashed-diamonds) as the dimension $n$ increases.}\label{fig:spec_accu}
\end{figure}

Finally, we show our main result, 
that --in exact arithmetic-- Algorithms \ref{simple} and \ref{algo:diag2trid} are
equivalent, which allows
us to conclude that our Algorithm  \ref{algo:diag2trid}, {\tt diag2trid},
is nothing but a stable 
implementation of the naive Algorithm \ref{simple}.

\subsection{\bf Schmeisser's algorithm.} 
Let $u(x)=x^n+a_{n-1}x^{n-1}+\ldots+a_1x+a_0$ be a real polynomial having only real
roots $\lam_1< \lam_2< \ldots < \lam_n$. Schmeisser in \cite{schmeisser1993real} 
proposed an algorithm to construct a symmetric, unreduced,
tridiagonal matrix $T\in\R^\nxn$ whose 
characteristic polynomial $p_n(x)=\det(xI-T)$ is equal to $u$. 

\begin{rem}
In fairness, in \cite{schmeisser1993real}, the author did not require, or need, the
eigenvalues to be distinct, and of course in the case of equal eigenvalues 
Schmeisser obtained a reduced tridiagonal form.  In our case, we are only
interested in distinct eigenvalues.
\end{rem}

For a polynomial $p(x)$ of degree $k\ge 0$, denote by $\gamma(p)$ its 
leading coefficient, so that $p/\gamma(p)$ is monic. The basic
algorithm proposed by Schmeisser goes as follows.

\begin{algorithm}
 \caption{Schmeisser's Algorithm}\label{algo:schm}
\begin{algorithmic}[1]
\State Set $f_1(x)=u(x)$, $f_2(x)=\frac{1}{n}u'(x)$.
\For{$\nu=1,\ldots,n-1$}
\State Divide $f_\nu(x)$ by $f_{\nu+1}$ to obtain $f_\nu(x)=q_\nu(x)f_{\nu+1}(x)-r_\nu(x)$.
\State Define $f_{\nu+2}(x)=r_\nu(x)/\gamma(r_\nu)$.
\EndFor
\State Set $q_n(x)=f_n(x)$ and $f_{n+1}(x)=1$.
\end{algorithmic}
\end{algorithm}

Then, the matrix $T$ is defined as:

\begin{equation}\label{eq:T_schm}
T=\bmat{
-q_1(0) & \sqrt{\gamma_1} & & \\
\sqrt{\gamma_1} & -q_2(0) & \sqrt{\gamma_2} & \\ \\
 & \ddots & \ddots & \ddots \\ \\
& &  \sqrt{\gamma_{n-2}} & -q_{n-1}(0) & \sqrt{\gamma_{n-1}} & \\
& & & \sqrt{\gamma_{n-1}} & -q_n(0)
},
 \end{equation}
 where $\gamma_k=\gamma(r_k)$ for all $k=1,\ldots, n-1$.
 
Note that, combining steps 2-3, one can express Algorithm \ref{algo:schm} as:
 
\begin{equation}\label{eq:schm_algo_comb}
\left\{
\begin{array}{l}
 f_\nu(x)=q_\nu(x)f_{\nu+1}-\gamma_\nu f_{\nu+2}(x),\ \nu=1,2,\ldots,n-1,\\
q_n(x)=f_n(x), \\
f_{n+1}(x)=1  .
\end{array}
\right.
\end{equation}

\subsection{\bf Our Algorithm: {\tt diag2trid}.} Given $n$ real numbers
$\lam_1< \lam_2< \ldots < \lam_n$, 
Algorithm {\tt diag2trid} below provides a simple and stable way to construct a real 
symmetric, unreduced, tridiagonal matrix $T$ having eigenvalues
$\lam_1,\ldots,\lam_n$.   

\begin{algorithm}[H]
 \caption{{\tt diag2trid}  Algorithm}\label{algo:diag2trid}
\begin{algorithmic}[1]
\State Set $D=\diag(\lam_1,\ldots,\lam_n)$.
\State Let $Q$ be a Householder reflection such that $Qe_1=e/\sqrt{n}$.
\State Set $A=Q^TDQ$.
\State Perform the \emph{Householder tridiagonalization} algorithm on $A$, so that
$$H^TAH=\widehat T\ , $$
where $\widehat T$ is tridiagonal and $H$ is orthogonal.
\end{algorithmic}
\end{algorithm}

For a complete description of the \emph{Householder tridiagonalization} algorithm, we 
refer to the book \cite{GVL2013matrixcomp} (see Algorithm 8.3.1). It is important 
to point out that the transformation $H$ is such that $H(:,1)=H(1,:)^T=e_1$. This means 
that $\widehat T$ is diagonalized by $(QH)^T$, whose first row is given by 
$e^T/\sqrt{n}$. Now, invoking the Implicit $Q$ Theorem
(see \cite[Theorem 8.3.2]{GVL2013matrixcomp}), we have that a real 
symmetric tridiagonal matrix $T$ is completely characterized by its (real) eigenvalues 
and by the first row of the (orthogonal) matrix of its eigenvectors, in the following sense.

\begin{lemma}\label{lem:impQtheo} Suppose $S$ and $T$ are $\nxn$ real symmetric, 
tridiagonal and unreduced, and such that 
$U^TSU=W^TTW=\diag(\lam_1,\ldots,\lam_n)$, with $U$ and $W$ orthogonal. If 
$U(1,:)=W(1,:)$, then $S_{ii}=T_{ii}$ for $i=1,\ldots,n$ and 
$\abs{S_{i, i-1}}=\abs{T_{i, i-1}}$ for $i=2,\ldots,n$. 
\end{lemma}

\begin{rem}
We note that with Algorithm \ref{algo:diag2trid}, we obtain the same matrix $\hat T$
regardless of how we initially arranged the eigenvalues on the diagonal matrix $D$.
This is because we have chosen $Q$ with first column given by $e/\sqrt{n}$, and
$Pe=e$ for any permutation matrix $P$.  A fortiori, in light of the equivalence
(in exact arithmetic) between Algorithms \ref{simple} and \ref{algo:diag2trid},
this had to be so, since Algorithm \ref{algo:schm} renders a tridiagonal $T$
from the sole knowledge of the characteristic polynomial.
\end{rem}

\subsection{\bf Mathematical equivalence between Algorithms \ref{simple}
and \ref{algo:diag2trid}.} 
Next, we show that, in exact arithmetic,
the matrix $T$ produced by Algorithm \ref{simple} and the matrix $\widehat T$ 
produced by Algorithm \ref{algo:diag2trid} are ``essentially" the same 
(i.e., up to the signs of the off-diagonal entries). 
Because of Lemma \ref{lem:impQtheo}, it is enough to show 
that $T$ in \eqref{eq:T_schm} is diagonalized into $D$ by an orthogonal matrix 
whose first row is equal to $e^T/\sqrt{n}$.

We will use the following fact from \cite[Equation (4)]{DPTZ2022eigenvec}. 

\begin{lemma}\label{lem:DPTZeigenvec}
If $A\in\C^\nxn$ is Hermitian with eigenvalues $\lam_1,\ldots,\lam_n$ and corresponding 
eigenvectors $v_1,\ldots,v_n$, then we have
\begin{equation*}
 \abs{(v_i)_1}^2 p'_{A}(\lam_i)=p_{M_1}(\lam_i),\ i=1,\ldots,n,
\end{equation*}
where $p_A$ is the characteristic polynomial of $A$ and $M_1=A(2:n,2:n)$. 
\end{lemma}

\begin{theorem} Let $\lam_1<\ldots<\lam_n$ and let $T\in\R^\nxn$ the symmetric 
tridiagonal matrix produced by Algorithm \ref{algo:schm}. Let $W\in\R^\nxn$ be 
orthogonal and such that $W^TTW=\diag(\lam_1,\ldots,\lam_n)$. Then 
$W(1,:)=e/\sqrt{n}$.
\end{theorem}
\begin{proof}
 Because of Lemma \ref{lem:DPTZeigenvec}, it is enough to show that
 $$
 p_{T(2:n,2:n)}(x)=\frac{1}{n}p'_T(x),\ \text{ for all } x\in\R.
 $$
 For any symmetric tridiagonal unreduced matrix
 $$T=
 \bmat{a_1 & b_2 & \\
b_2 & a_2 & b_3 \\
& \ddots & \ddots & \ddots \\
& & b_{n-1} & a_{n-1} & b_{n} \\
& & & b_{n} & a_{n} \\
 },
 $$
the following recursive formula involving the determinants of its principal submatrices is well known:
\begin{equation*}
\left\{
\begin{array}{l}
p_0(x)=1\\
p_1(x)=(x-a_1) \\
p_k(x)=(x-a_k)p_{k-1}(x)-b_k^2p_{k-2}(x),\ k=2,3,\ldots,n
\end{array},
\right.
\end{equation*}
where $p_k(x)=\det(T(1:k,1:k))$  for all $k\ge 2$.

Equivalently, one can run across the diagonal of $T$ from bottom to top, and rewrite the recursion as
\begin{equation}\label{eq:recursive_rewritten}
\left\{
\begin{array}{l}
f_{n+1}(x)=1 \\
f_n(x)=(x-a_n) \\
f_k(x)=(x-a_k)f_{k+1}(x)-b_{k+1}^2f_{k+2}(x),\ k=n-1, n-2, \ldots,1
\end{array}
\right.
\end{equation}
where $f_k(x)=\det(T(k+1:n,k+1:n))$ for all $k\le n-1$.

Notice that, fixing
\begin{equation*}
\left\{
\begin{array}{l}
q_\nu= \ x-a_\nu\\
\gamma_\nu= \ b_\nu^2
\end{array},\ \nu=1, 2, \ldots, n-1
\right.
\end{equation*}

the recurrence formula \eqref{eq:recursive_rewritten} is equivalent to \eqref{eq:schm_algo_comb}, and hence
$p_{T(2:n,2:n)}(x)=\frac{1}{n}p'_T(x)$ by construction. 
\end{proof}

\end{document}